\documentclass[12pt,english]{amsart}
\usepackage[T1]{fontenc}
\usepackage[utf8]{luainputenc}
\usepackage{xcolor}
\usepackage{pdfcolmk}
\usepackage{amsthm}
\usepackage{amsmath}
\usepackage{stmaryrd}
\usepackage{stackrel}
\usepackage{setspace}
\PassOptionsToPackage{normalem}{ulem}
\usepackage{ulem}
\usepackage{tikz}
\makeatletter

\providecolor{lyxadded}{rgb}{0,0,1}
\providecolor{lyxdeleted}{rgb}{1,0,0}

\numberwithin{figure}{section}
\numberwithin{equation}{section}
\numberwithin{table}{section}

  \input amssymb.sty
\usepackage{etoolbox}
\patchcmd{\thebibliography}{\section*}{\section}{}{}

\usepackage{algpseudocode} \usepackage{algorithm}
\usepackage{algpseudocode,algorithm}
\usepackage{lipsum}
\usepackage{caption}
\usepackage{graphics}

\usepackage{amsmath}
\usepackage{amssymb}
\usepackage[all]{xy}

\usepackage{epsfig}\usepackage{youngtab}

\newcommand{\ef}{\end{equation}}
\chardef\bslash=`\\ 

\newdir{:=}{{}}
\newcommand*\colvec[3][]{
    \begin{pmatrix}\ifx\relax#1\relax\else#1\\\fi#2\\#3\end{pmatrix}
}

\newtheorem*{thm*}{Theorem}

\newtheorem{lem}{Lemma}[section]
\newtheorem*{lem*}{Lemma}

\newtheorem*{corl*}{Corollary}
\newtheorem{prop}{Proposition}[section]

\newtheorem{prop*}{Proposition}
\newtheorem{Goal*}{Goal}

\theoremstyle{definition}
\newtheorem{defn}{Definition}[section]
\newtheorem{examp}{Example}
\newtheorem*{examp*}{Example}
\newtheorem*{remark*}{Remark}
\newtheorem*{CC*}{Crossover Conjecture}
\newtheorem*{Note*}{Note}
\newtheorem*{defn*}{Definition}

 \theoremstyle{remark}
\newtheorem{remark}{Remark}[section]

 \renewcommand{\sectionmark}[1]{}

\newcommand{\la}{\langle}
\newcommand{\ra}{\rangle}

\newcommand{\thrh}{\frac{3}{2}}

\newcommand{\sevh}{\frac{7}{2}}
\newcommand{\eleh}{\frac{11}{2}}

\newcommand{\defect}{\operatorname{def}}

\newcommand{\diag}{\operatorname{diag}}

\newcommand{\lm}{\lambda}
\newcommand{\Lm}{\Lambda}

 \newcommand{\supp}{\operatorname{supp}}

\newcommand{\hub}{\operatorname{hub}}
\newcommand{\cont}{\operatorname{cont}}

\renewcommand{\a}{\alpha}

\makeatother

\usepackage{babel}

\begin{document}
\title{Symmetry for the spin highest weight module   $V(\Lambda_n)$}

\author{Ola Amara-Omari}
\address{ Bar-Ilan University, Ramat-Gan, Israel}
\email{Omariao@biu.ac.il}

\maketitle

 \begin{abstract} 
 For an affine Lie algebra $\mathfrak g$ of type $A^{(2)}_{2n}$ we get a special set of partitions for the highest weight modules $V(\Lambda)$, where $\Lambda=\Lambda_n$, and we also get a some results for the canonical basis elements. We demonstrate a symmetry of weights with fixed $n $-content around the weights gotten from the empty partition by adding multiples of the null root. We also find the highest power of the quantum parameter $q$ in an important set of canonical basis elements.
 
\end{abstract}

\maketitle
\section{INTRODUCTION}
The projective representations of the symmetric groups $S_n$ are representations of a central extension $\tilde{S_n}$ of the symmetric group by a cyclic group of order two. When the generator $z\in Z(\Tilde{S_n})$ satisfies $\phi (z)=1$, the representation is called ordinary and for $\phi (z)=-1$ we get a spin representation. The $p$-strict partitions labeling spin representations represent the irreducible modules over a field of characteristic $p$ [MY1,MY2]. Leclerc and Thibon [LT] conjectured that the spin representation for  affine Lie algebra corresponds to the simple modules of highest weight representation $V(\Lambda)$ for the case $\Lambda=\Lambda_0$, and it was proven in [BK]. In our last paper [AS1] we conjecture an algorithm to construct spin multipartitions for general highest weight modules  $V(\Lambda)$ in $A^{(2)}_{2n}$ , where $h=2n+1$.

In this paper we construct partitions for the case $\Lambda=\Lambda_n$, and get  partitions of special form. Also we get some result for the canonical basis elements and we define a slice in the crystal corresponding to this case. For a Kac-Moody algebra $\mathfrak g$ of classical affine type one can construct a $q$-deformed Fock space of level 1, and in [AS1] it was shown that the Fock space is a module for the quantum group $U_q(\mathfrak g)$ in the case where  $\mathfrak g$ is of type $A^{(2)}_{2n}$. Leclerc and Thibon [LT] studied this Fock space for $\Lambda=\Lambda_0$ and get a connection with spin representations of $\mathfrak S_m$ in characteristic $h$, where $h$ is prime, such that the action of the standard generators of $U_q(\mathfrak g)$ corresponds to Morris branching rules describing induction and restriction of  spin representation between $\mathfrak S_m$ and $\mathfrak S_{m+1}$. We study the case  $\Lambda=\Lambda_n$, giving information on the structure of the canonical basis elements in important cases. 
   
\section{DEFINITIONS AND NOTATION}

Let $\mathfrak g$ be the affine Lie algebra $A^{(2)}_{2n}$ as in [\cite{Ka}, Ch. 4] and let $C=[a_{ij}]$ be the Cartan matrix. For $n=1$, the Cartan matrix is
\[
C=
\begin{bmatrix}
2&-4\\
-1&2
\end{bmatrix}.
\]
For the general case $n\geq 2$, we have
\[C=
\begin{bmatrix}
2&-2&0&\dots&0&0\\
-1&2&-1&0&\dots&0\\
0&-1&2&-1&\dots&0\\
\dots&\dots&\dots&\dots&\dots&\dots\\
\dots&\dots&\dots&\dots&\dots&\dots\\
0&\dots&0&-1&2&-2\\
0&0&\dots&0&-1&2
\end{bmatrix}.
\]

 Let $Q$ be the $\mathbb Z$-lattice generated by the simple roots
$\a_0,\dots,\a_{n}$. Kac \cite{Ka} uses the notation $\a_i^\vee$ for the coroots, but we will denote the corresponding coroots by $h_0, h_1, \dots h_\ell$. 
 The pairing between a root and a coroot is given by the appropriate entry in the Cartan matrix. 
\[
\langle h_i, \a_j \rangle =a_{ij}, i,j=0,1,\dots,n.
\]

The Cartan matrix is not symmetric, but it is symmetrizable, the symmetrizing matrix being the diagonal matrix $D=\diag(\frac{1}{2}, 1,\dots, 1, 2)$. Thus in the symmetric product, $(\a_0\mid \a_0)=1$,
	$(\a_i \mid \a_i)=2$ for $i=1,2,\dots,n-1$ and $(\a_n \mid \a_n)=4$.
Other values of this symmetric product of interest to us will be $(\a_0\mid\a_1)=-1$, $(\a_{n-1}\mid\a_n)=-2$. Also $(\Lambda_i \mid \a_j)=\delta_{ij}$, for all
$i,j=0,1,\dots,\ell$.
In this paper, we will be working with a fixed non-zero
 dominant integral weight, that is to say, a weight $\Lambda$ such that $\langle h_i, \Lambda \rangle $ is a nonnegative integer for every $i$.  

The Cartan subalgebra of the Lie algebra $\mathfrak g$ is generated by the 
$h_0, h_1, \dots, h_\ell$ and an element $d$ called the \textit{scaling element},
which satisfies 
\[
\langle d, \alpha_i \rangle =0, 1 \leq i \leq \ell, \langle d,\alpha_0 \rangle =1.
\]

The center of  $\mathfrak g$ is one-dimensional, spanned by 
\[
c = h_0+2\sum_{i=1}^n h_i,
\] 
which is called the \textit{canonical central element}. The \textit{level} $r$ of a weight $\lambda$ is $\langle \lambda, c \rangle$. 

Let  $Q_+$ be the subset of $Q$ in which all coefficients are nonnegative.
The \textit{null root} $\delta$ is $2\a_0+2\a_1\dots+2\a_{n-1}+\a_{n}$.
This formula is chosen for the null root because the vector $(2,2,2,\dots,2,1)$, as a column vector, produces $0$ when multiplied by the Cartan matrix on the left. We have  $(\a_i\mid \delta)=0$. 

 We define the fundamental weights $\Lambda_j, 0 \leq j \leq \ell$ together with the null root  to be weights dual to the coroots and the scaling element, so chosen that $ \langle h_i, \delta \rangle = 0$ for all $i$, $\langle d, \delta \rangle=1$, $ \langle d, \Lambda_j \rangle = 0$ for all $j$, and
$ \langle h_i, \Lambda_j \rangle = \delta_{ij}$, where $\delta_{ij}$ is the Kronecker delta, non-zero only when $i=j$ and then equal to $1$.

The $\mathbb Z$-lattice $P$ of weights of the affine Lie algebra spans a real vector space that has two different bases.  One is given by t
he fundamental weights together with the null root, $\Lambda_0,\dots, \Lambda_{\ell}, \delta$, 
and one is given by $\Lambda_0, \a_0,\dots,\a_{\ell}$.  We will usually use the first basis for our weights, where Kac \cite{Ka} prefers the second.
 In addition, $(\Lambda_0\,\mid \Lambda_0 )=(\delta \mid \delta) =0$ and $(\Lambda \mid \delta )=r$, the level defined in the previous paragraph.

 Let $\Lambda$ be a fixed non-zero dominant integral weight that is a sum of fundamental weights with nonnegative integer coefficients. Let $V(\Lambda)$ be a highest weight module with  
 highest weight $\Lambda$, and let $P(\Lambda)$ be the set of weights of $V(\Lambda)$. The Dynkin diagram for affine type $A^{(2)}_{2n}$ is a graph with vertices $I=\{0,1,\dots,n\}$ and connecting adjacent verticies, double arrow at each end pointing toward zero. In affine type $A^{(2)}_{2n}$, if $\Lambda=a_0\Lambda_0+a_1\Lambda_1+\dots+ a_n\Lambda_n$ for nonnegative integers $a_i$, the level is given by the formula $r=a_0+2a_1+ \dots +2a_n$ . We have $r>0$ because $\Lambda$ is non-zero. 
 \begin{defn}
 Let $t=a_0+a_1+\dots+a_n$. A $multicharge$ is a t-tuple $s$ of integers. In our work it is always $a_0$ copies of the zero, $a_1$ copies of 1, up to $a_n$ copies of $n$, and we write $s=(k_1,k_2,\dots,k_t)$. 
 \end{defn}

  As in [\cite{Kl},\S3.3], we define the defect of a weight $\lambda=\Lambda -\alpha$ by 

\[
\defect(\lambda)=\frac{1}{2}((\Lambda \mid \Lambda)-  (\lambda \mid \lambda))=(\Lambda \mid \alpha)-\frac{1}{2}(\alpha \mid \alpha).
\]
\noindent
Since we are in a highest weight module, we always have $(\Lambda \mid \Lambda) \geq (\lambda \mid \lambda)$ [\cite{Ka}, Prop. 11.4(a)].  The definition of the defect is general for all affine types, and our affine Lie algebra is the only one for which the defect may be a half-integer. There is an alternative definition of the defect which multiplies it by $2$ to get an integer, but we prefer to use the general definition. The weights of defect $0$ are those lying in the Weyl group orbit of $\Lambda$.
Every weight $\lambda \in P(\Lambda)$ has the form $\Lambda-\alpha$, for $\alpha \in Q_+$. We will need
\begin{align*}
\defect(\lambda-j \delta)=&(\Lambda \mid \alpha+j\delta)-\frac{1}{2}(\alpha +j\delta \mid \alpha +j\delta)\\
=&\defect(\lambda)+jr.
\end{align*}

\begin{defn}\label{deg} Let $\lambda = \Lambda - \alpha$ with $\alpha \in Q_+$. If $\alpha=\sum_{i=0}^\ell c_i \alpha_i$, where all $c_i$ are nonnegative, then the vector 
\[
\cont(\lambda)=(c_0.\dots,c_\ell)
\]
is called the \textit{content} of $\lambda$. The \textit{degree} $\deg(\lambda)$ of the weight $\lambda$ is the integer $n=\sum_{i=0}^\ell c_i$.
\end{defn}

Define
\[
\max (\Lambda)=\{\lambda \in P(\Lambda) \mid \lambda + \delta \not\in P(\Lambda)\},
\] 
and by [\cite{Ka}, \S 12.6.1], every element of 
$P(\Lambda)$ is of the form $\{y-j\delta \mid y \in \max(\Lambda), j \in \mathbb Z_{ \geq 0}\}$.

\begin{defn}\label{M} Let $W$ denote the Weyl group
	\[
	W=T \rtimes  \mathring{W}
	\] 
	expressed as a semidirect product of 
	a normal abelian subgroup $T$ by the finite Weyl group $\mathring{W}$ given by crossing out the first row and column of the Cartan matrix.  The elements of $T$ are transformations of the form 
	
	\[
	t_\alpha(\zeta)=\zeta+r\alpha -((\zeta|\alpha)+\frac{1}{2}(\alpha|\alpha)r)\delta,
	\]
	\noindent where for affine type A, the weights $\a$ are taken from the $\mathbb Z$-lattice $M$ generated by $\a_1, \dots ,\a_{\ell}$, omitting $\a_0$ [\cite{Ka}, \S 
	6.5.2] but in the case $A^{(2)}_{2n}$, from a lattice generated by the long roots with half integer coefficients. The role of $M$ is explained in detail in \cite{FSv2}, the version of \cite{BFS} which treats the spin case in greatest detail. In particular,  Prop. 2.11 of \cite{FSv2} implies that any two positive hubs in $P(\Lambda)$ differ by an element of $Q$.
\end{defn}

\begin{defn}\label{hub} For any weight $\lambda$ in the set of weights  $P(\Lambda)$ for a dominant integral weight $\Lambda$, we let $\hub(\lambda)=[\theta_0,\dots, \theta_{\ell}]$ be the \textit{hub} of $\lambda$, where 
	\[
	\theta_i=\la h_i, \lambda \ra.
	\]
	\noindent Let $U$ be the subspace of the real weight space generated by the fundamental weights. The hub is the projection (with respect to the decomposition $U \oplus \mathbb{R}\delta$) of the weight of $\lambda$ onto $U$. We write the hub with square brackets to distinguish the hub from the content. The original definition given by Fayers in \cite{Fa} was the negative of this one.
\end{defn} 
\section {partitions with $\Lambda=\Lambda_n$}

A conjectured algorithm for general restricted spin multipartitions is given in [AS1]. Constructing the highest weight module for spin multipartitions involves showing that each $V(\Lambda_i)$ is a module for the quantum enveloping algebra and then taking a tensor product. The case   $\Lambda=\Lambda_n$ had many special properties.

\begin{defn}\label{3.1}
	The set $P(\Lambda)$ can be taken as the set of vertices of a graph $\widehat P(\Lambda)$ which we will call the \textit{block-reduced 
		crystal}. Two weights $\mu,\nu \in P(\Lambda)$ will be connected by an edge of residue $i$ for $i=0,1,\dots, n$   if $\mu -\nu = \pm \alpha_i$ for a simple root $\alpha_i$. A maximal set of vertices connected by edges of residue $i$ will be called an \textit{$i$-string}
	 and all $i$-strings are finite because $V(\Lambda)$ is integrable.
\end{defn}
The block-reduced crystal was used in  [\cite{AS}, Def. 2.6] or [\cite{BFS},\S 2]. 
	\begin{examp}
		As an example, consider the case of $h=3$ and $\Lambda=\Lambda_1$.  The reduced crystal $\widehat P(\Lambda)$ is given in Figure 1 to degree 16.  The vertices are labeled by hubs, with the defect as an exponent.  The action of $f_1$ is diagonal to the right, and of $f_0$ is diagonal to the left. The fundamental weight $\Lambda_1$ is positive integral with hub $[0,1]$, so the level is
		$r=(\Lambda \mid c) = 2$. For any hub $[\theta_0,\theta_1]$ we must have $r=\theta_0+2\theta_1$.
		
		\begin{tikzpicture}
			
			\draw (0,0) node[anchor=south]{$[0,1]^0$} -- (.5, -.5) node[anchor=west]{$[4,-1]^0$}; 
			\draw  (.5,-.5)   --  (-1.5, -2.5)  node[anchor=east]{$[-4,3]^0$};
			\draw (0,-1) node[anchor=east]{$[2,0]^{\thrh}$}; 
			\draw (-.5,-1.5) node[anchor=east]{$[0,1]^2$} -- (0,-2) node[anchor=west]{$[4,-1]^2$};;
			\draw (-1,-2) node[anchor=east]{$[-2,2]^\thrh$} -- (0,-3) node[anchor=west]{$[6,-2]^\thrh$} ;
			\draw (-1.5,-2.5) -- (0, -4) node[anchor=west]{$[8,-3]^0$};
			\draw (0,-2) -- (-2,-4)	node[anchor=east]{$[-4,3]^2$};	
			\draw (0,-3) --(-3,-6) node[anchor=east]{$[-6,4]^\thrh$};
			\draw (-1.5,-3.5) node[anchor=east]{$[-2,2]^\sevh$} -- (-.5,-4.5) node[anchor=west]{$[6,-2]^\sevh$};
			\draw (0,-4) --(-4,-8) node[anchor=east]{$[-8,5]^0$};
			\draw (-2,-4) -- (-.5,-5.5) node[anchor=west]{$[8,-3]^2$};
			\draw (-.5,-5.5) -- (-3,-8);
			\draw (-2,-5) node[anchor=east]{$[-2,2]^\eleh$} -- (-1,-6) node[anchor=west]{$[6,-2]^\eleh$};
			\draw (-2.5,-5.5) node[anchor=east]{$[-4,3]^4$} -- (-1,-7) node[anchor=west]{$[8,-3]^4$};
			\draw (-1,-7) -- (-2,-8);
			\draw (-3,-6) --(-1,-8) node[anchor=west]{$[10,-4]^\thrh$};
			\draw (-3,-7) node[anchor=east]{$[-4,3]^6$} -- (-2,-8);
			\draw (-3.5,-7.5) node[anchor=east]{$[-6,4]^\sevh$} -- (-3,-8);
		\end{tikzpicture}
		
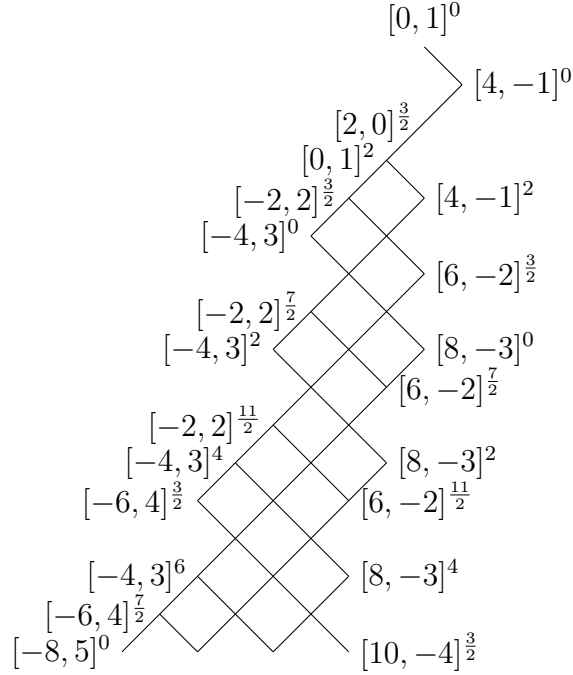
		Figure 1. $\Lambda=\Lambda_1 $, $h=3$. Vertices are marked by hub and defect.
		
\end{examp}

 We now make use of the hub and its components $\theta_i$ from Def \ref{hub}. Let us first recall the residue function we are using, following \cite{LT}. We have an odd integer $h=1+2n$. 

\begin{defn}
    For any integer $m$ we define $\widehat{m}$ to be min$ \{ m-1 \mod h,-m \mod h\}$ . 
\end{defn}

Setting aside for the moment residue $0$, we try a similar procedure taking the spin residue of $k_\ell+s-r+1$. We have to add $1$ because $\hat{1}=0$ and  $\hat{2}=1$. We need 
the spin residue of the node $(1,1,\ell)$ to be $k_\ell$
and the spin residue of $(1,2,\ell)$ to be $k_\ell+1$.  

\begin{examp}
    Let $\Lambda=\Lambda_3,h=7$, we could get a Young diagram with residues

    \young(321001232,2321,12,01,00)
\end{examp}

\begin{defn}
	A partition  is $h$-\textit{restricted} if it does not have 
	$h$ identical columns and is $h$-strict in each partition has no identical rows unless the length of the row is divisible by $h$.
\end{defn}

\begin{defn}
	Let $\mathfrak{g}= A^{(2)}_{2n}$ and let $\Lm=(a_0\Lm_0+\dots+a_n\Lm_n)$ be a positive integral weight for $\mathfrak{g}$. A \textit{spin multipartition} for $\Lm$ is a set of $a_0+a_1+\dots+a_n$ partitions, of which
	\begin{itemize}
		\item  The bottom $a_0$ are $h$-restricted $h$-strict partitions each of which has $0$ in its upper left-hand corner when we use the English convention for Young diagrams of a partition. These will be called the  $0$-\textit{corner partitions}.
		\item For each $i>0$ in ascending order, $a_i$ partitions which are $h$-restricted but not need to be $h$-strict.  The residue of a node with row $s$ and column $t$ will be $\widehat{(i+t-s+1)}$, and the partition will be called an $i$-\textit{corner partition}. If $0<i<j\leq n$, the $i$-corner partitions will lie below the $j$-corner partitions. Given our convention for the multicharge $s=(k_1,\dots,k_t)$, the residue of the node $(t,s,\ell)$ is $\widehat{(k_{\ell}+t-s+1)}$.
	\end{itemize}
\end{defn}

\begin{defn} 
	Given a  $j$-corner partition $\lambda$, for a residue $j>0$, and an arbitrary residue $i$,  an $i$-node adjacent to the end of a row of $\lm$ is an \textit{addable $i$-node} if adding it produces an well-defined partition, and there are also double addable $0$-nodes. The complete set of addable $i$-nodes is the maximal set of $i$-nodes such that adding them produces an well-defined partition. Similarly,  a node on the end of a row of a partition is \textit{$i$-removable} if  removing the node still gives an well-defined partition. The complete set of removable $i$-nodes is the set of all $i$-nodes in $\lambda$ that can be removed and will leave an well-defined partitions.
\end{defn}
\textbf{Conjectured algorithm for contructing $h$-regular $h$-restricted spin multipartitions.} We designate addable and removable nodes  by $+$ and $-$ as Kleshchev does \cite{Kl}  and write them from left to right as we go up in the multipartition, to get the \textit{$i$-signature}. In $i$-corner partitions for $i \neq 0$, it may happen that a removable $0$-node lies directly above an addable $0$-node.  Such a pair of nodes will be designated in the list by $-+$. 
	We eliminate any cases of $+-$ from the list. 
When we have made all the eliminations, we have a string of $-$, and finally a string of $+$.  This list is called the \textit{reduced} $i$-signature. Any or all of the strings could be empty, but if there is a $+$, then the leftmost is the \textit{i-cogood} node which will be 
added to produce the result of operating by $f_i$ in the Kashiwara crystal.

The set of all the partitions obtained by acting on the empty partitions by various $ f_i$, will be
called the $h$-restricted partition. It remains to prove that the $h$-restricted partitions are in one-to-one correspondence with
basis elements of a highest weight module $V(\Lambda)$.

The following Lemma is responsible for most of the special properties of $\Lambda=\Lambda_n$.

\begin{lem}
Let $\Lambda=\Lambda_n, h=2n+1,$ and let $\lambda \in P(\Lambda)$ be a partition, then the residue of the node  (i,j) in  $\lambda$ is equal the residue of the node  (j,i) in  $\lambda$.
\end{lem}

\begin{proof} The multicharge in this case is $s=(n)$, so there is a single partition and $k_1=n$.

According to the algorithm we construct in [AS1], the residue of the node $(i,j)$ in $\lambda$ is equal to  $(\widehat{k_\ell+j-i+1})$, and in our case is equal to $\widehat{(n+j-i+1)}$ , where the hat is over the entire sum. We now recall that this residue equals $\equiv min \{-(n+(j-i)+1) \mod h ,(n+(j-i)) \mod h \}$. Substitute $n=\frac{h-1}{2}$ in the first argument and add $h$ we get $ min \{n+(i-j) \mod h ,n+(j-i) \mod h \}$. Now for the residue of the node $(j,i)$ in $\lambda$ is equal to  $(\widehat{k_\ell+i-j+1})$=  $\widehat{(n+j-i+1)}$ that equal to $ min \{-(n+(i-j)+1) \mod h ,(n+(i-j)) \mod h \}$, Substitute $n=\frac{h-1}{2}$ in the first argument and add $h$ to it we get $ min \{\frac{h-1}{2}+(i-j) \mod h ,n+(j-i) \mod h \}$ $\equiv min \{n+(i-j) \mod h ,n+(j-i) \mod h \}$. so we get the same arguments in the two nodes. 
\end{proof}

As we explained in the introduction, the spin multipartitions require a more complicated algorithm than the type A $e$-restricted multipartitions because we are taking them from a tensor product in which the factors are mutually non-isomorphic.  We could have found ourselves in a situation where each residue requires a different criterion.  Fortunately, this does not appear to be the case.


\begin{lem}
Let $\Lambda=\Lambda_n, h=2n+1,$ and let $\lambda$ be a partition such that $ \defect(\lambda)=0$, then $\lambda^t=\lambda$.
\end{lem}

\begin{proof}
Every partition $\lambda \in P(\Lambda)$ such that $ \defect(\lambda)=0$ we can get by $s_{i_m}\dots s_{i_2}s_{i_1}(\emptyset)$. By induction on the degree  $m$, for $m=1 , s_n(\emptyset)=\young(n)$. \\\ For $m=2 , s_{n-1}s_n(\emptyset)=\young({n}{i},{i}), i=n-1$ .\\  Assume that after $m$ steps we have the partition $\mu$ such that $\defect(\mu)=0$, $\mu^t=\mu$ and $\hub (\mu)=[\theta_0,\dots,\theta_n]$, and let us prove for $m+1$. Now we operate by $s_j$, and we get $\defect(f_j^{(\theta_j)}(\mu))=0$, now let $f_j^{(\theta_j)}(\mu)=\lambda$, and $\lambda_k $ be the length of the row $k$ in the partition $\lambda$. Assume that $(k,s)$ is a $j$-addable node in  $\lambda$, since that $\lambda_{k+1} \leq \lambda_k \leq \lambda_{k-1} $ and $\lambda^t=\lambda$ we get $\lambda^t_{k+1} \leq \lambda^t_k \leq \lambda^t_{k-1} $, the node $(s,k)$ is also an $j$-addable node on $\lambda$. According to Lemma 3.1, $(s,k)$ is a j-node, since the $j$-nodes are symmetric. 
\end{proof}

	\section{Slices in $\Lm=\Lm_n$}

		Let $I_0$ be a set of residues, $I_0\subset I$. We give a realization of the graph $\widehat P(\Lambda)$ by embedding it in $P_\mathbb{R}$.
		
		\begin{defn}\label{face} 
		 	An $I_0$-\textit{slice} is the intersection of $\widehat P(\Lambda)$ with a linear subspace $V$  on which, for all $j \in I-I_0$, there is a nonnegative integer $d_j$ such that   $v_j=d_j$ for all vectors $v \in V$. A $I_0$-\textit{face} is a slice containing the highest weight $\Lambda$, so that all the fixed integers $d_j$ are $0$.
		\end{defn}	
	
	\begin{remark} Any $I_0$-face is contained in $\max(\Lambda)$ and thus by the results of \cite{BFS} the weight is entirely determined by the hub. 
		
		If $I_0=I-\{n\}$, then a $I_0$-slice consists of all vertices with a fixed $n$-component $c_n$ of the content. In \cite{AS}, this was called a floor.
	\end{remark}

          \begin{figure}[h] 
				\begin{tikzpicture}
					\draw (0,0) node[anchor=south]{$[0,0,2,-1]^0$} -- (1, -1) node[anchor=west]{$[0,1,0,0]^1$};
					\draw (1,-1) -- (2,-2) node[anchor=west]{$[0,2,-2,1]^0$} ;
                    \draw (1,-1)-- (1,-2) node[anchor=east]{$[2,-1,1,0]^1$} ;
                    \draw (2,-2) -- (2,-3) node[anchor=west]{$[2,0,-1,1]^1$} ;
                    \draw (1,-2) -- (2,-3);
                   \draw (2,-3) -- (2,-4) node[anchor=west]{$[4,-2,0,1]^0$} ;
                   \draw (1,-2)-- (0,-3) node[anchor=east]{$[0,0,1,0]^{\frac{3}{2}}$} ;
                  \draw (2,-3) -- (1,-4) node[anchor=east]{$[0,1,-1,1]^{\frac{3}{2}}$} ;
                  \draw (0,-3) -- (1,-4) ;
                   \draw  (0,-3) -- (-1,-4) node[anchor=east]{$[-2,1,1,0]^1$} ;
                   \draw (1,-4) -- (0,-5) node[anchor=north]{$[-2,2,-1,1]^1$} ;
            \draw (-1,-4) -- (0,-5) ;
            \draw (0,-5) -- (0,-6) ;
            \draw (1,-4) -- (1,-5) ; 
            \draw (2,-4) -- (1,-5) node[anchor=west]{$[2,-1,0,1]^{\frac{3}{2}}$} ;
            \draw (1,-5) -- (0,-6);
            \draw (-1,-4) -- (-1,-5) node[anchor=east] {$[0,-1,2,0]^1$} ;
            \draw (-1,-5) -- (0,-6) node[anchor=west]{$[0,0,0,1]^2$} ;
            \draw (0,-6) -- (1,-7) node[anchor=west]{$[0,1,-2,2]^1$} ;
           \draw (0,-6) -- (0,-7) node[anchor=north]{$[2,-2,1,1]^1$} ;
              \draw (0,-6) -- (-1,-7)  node[anchor=east]{$[-2,1,0,1]^{\frac{3}{2}}$}  ;
                \draw (1,-7) -- (1,-8) node[anchor=west]{$[2,-1,-1,2]^1$} ;
              \draw (0,-7) -- (1,-8) ;
               \draw (0,-7) -- (-1,-8) node[anchor=west]{$[0,-1,1,1]^2$};
                \draw (-1,-7) -- (-1,-8) ;
                \draw (1,-8)-- (0,-9) node[anchor=west]{$[0,0,-1,2]^2$};
                \draw (-1,-8) -- (0,-9);
                 \draw (-1,-7) -- (-2,-8) node[anchor=east]{$[-4,2,0,1]^0$};
                 \draw (-1,-8)--(-2,-9);
                 \draw (-2,-8)  -- (-2,-9) node[anchor=east]{$[-2,0,1,1]^1$};
                 \draw (-2,-9)--(-1,-10) node[anchor=west]{$[-2,1,-1,2]^1$};
                 \draw (0,-9)-- (-1,-10);
                 \draw (-2,-9)--(-2,-10) node[anchor=east]{$[0,-2,2,1]^0$};
                 \draw (-1,-10)--(-1,-11) ;
                  \draw (-2,-10) -- (-1,-11) node[anchor=east]{$[0,-1,0,2]^1$} ;
                  \draw (-1,-11)--(0,-12) node[anchor=west]{$[0,0,-2,3]^0$}  ;
				\end{tikzpicture}
			   \centering
			   \caption{ $\Lambda=\Lambda_3$, $p=7$, $I_0=\{0,1,2\}$}
			\end{figure}
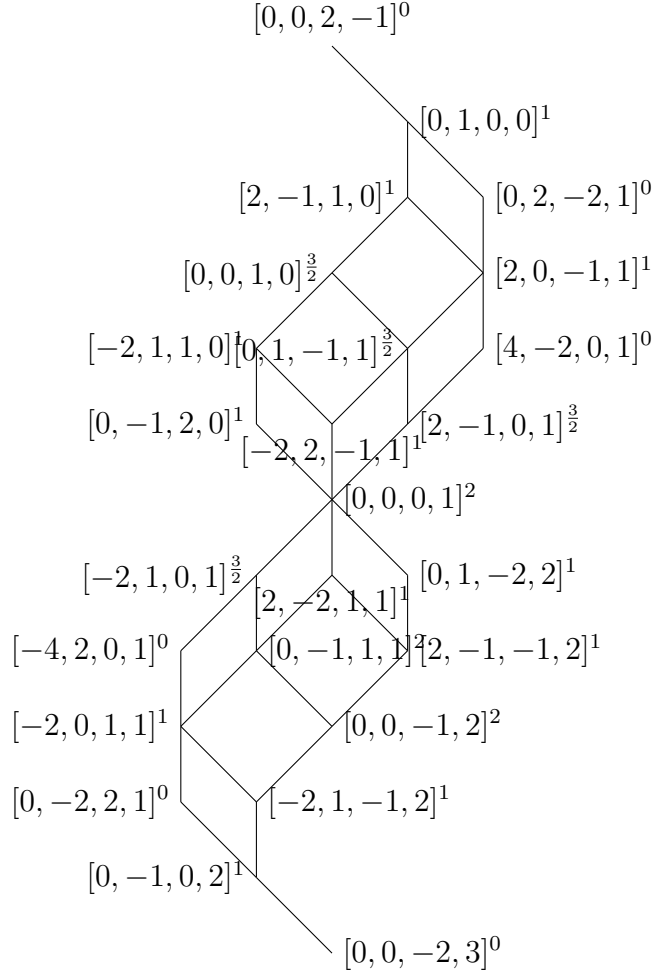
            
		A $[0,1,\dots,n-1]$-slice is isomorphic to a module for the finite Lie algebra of type $B_{n}$. First we want to prove a lemma that appears in \cite{H} as an exercise.

        \begin{lem}
        There is a unique element $w_0 \in W$, that is  called the longest element in $P(\Lambda)$. Its length is exactly the number of positive roots in $\Delta$,  and it sends all the positive roots to negative roots, and this element maps the fundamental Weyl chamber to its opposite, $w_0(\phi^+)= \phi^-=-\phi^+$.
         \end{lem}
\noindent
Sketch of proof:
		 The finite Lie algebra module has an involution in its hub coordinates which takes the highest weight element $\Lambda$  to its image $\rho$  under the longest element of the finite Weyl group, which  we will denote by $\sigma$. In terms of degrees in the crystal, $\Lambda$ has the lowest degree and corresponds to the empty multipartition, while $\rho$ has the highest degree, which is to say that its multipartition has the greatest number of nodes. The involution of the irreducible modules of the finite Lie groups is discussed in 
	 [\cite{H},\S 10, \S 21].
		 
		Specifically, there is a duality between the highest weight module $V(\lambda)$ for the highest weight $\lambda$ and $V(-\sigma \lambda)$, where $\sigma$ is the longest element of the finite Weyl group. The procedure uses the automorphism of the Dynkin diagram and sends a weight with hub $[b_1,b_2,\dots,\frac{r}{2}-\sum_{i=1}^{n-1}b_i]$ to a weight with hub  $[-b_1,-b_2,\dots,\frac{r}{2}+\sum_{i=1}^{n-1}b_i]$. The actual results we need are given as exercises, 10.9,13.5 and 21.6.  

\begin{prop}
Let $\Lambda=\Lambda_n,h=2n+1$, a slice for an interval $\{0,\dots.n-1\}$ has an involution $\tau$ which send the weight $\lambda$ with hub $[b_1,b_2,\dots,b_n]$ to  a weight $\lambda'$ with hub $[-b_1,-b_2,\dots,r-b_n]$ and satifies $\tau(f_i(\lambda))=e_i(\lambda')$, and the involution preserve the defect of the weights, $defect (\lambda)=defect (\lambda')$. 
\end{prop}

       \begin{proof}
           We will check that $\tau$ is an involution on the $n$ components of the hub.
           $\tau(\tau([b_1,b_2,\dots,b_n]))=\tau[-b_1,-b_2,\dots,r-b_n]=[b_1,b_2,\dots,b_n]$.
           Now by induction, let the proposition  hold on the  length $m$ of a path in $P(\Lambda)$ where the partition is $\lambda$, and $f_i(\lambda)=\mu$, the dual operation is acting on $\tau(\lambda)$ by $e_i$. Its enough to prove that $e_i(\tau(\lambda))=\tau(f_i(\lambda))$. In the $i$-component of $f_i(\lambda)$ we get $b_i-2$, and for $i \in \{1,\dots, n-1\}$ the $(i-1)$-component and the $(i+1)$-component rise by 1, so $\tau[b_0,\dots,b_{i-1}+1,b_i-2,b_{i+1}+1,\dots,b_n ]=[-b_0,\dots,-b_{i-1}-1,-b_i+2,-b_{i+1}-1,\dots,r-b_n ]$, in the other hand $e_i(\tau(\lambda))=[-b_0,\dots,-b_{i-1}-1,-b_i+2,-b_{i+1}-1,\dots,b_n ]$. Now when 
           $i=0$ we get $f_i(\tau(\lambda))=\tau(e_i(\lambda))=[-b_0+2,-b_1-2\dots,-b_{i-1},-b_i,-b_{i+1},\dots,r-d_n ]$
       \end{proof}

	\begin{prop}
 The  slice on $\Lambda=\Lambda_n,h=2n+1, c_n=c$ is symmetric around the hub $[0,\dots,0,1]^{2c}$.    
\end{prop}

\begin{proof}
The $V(\Lambda_n)$ module has a Kashiwara basis, and each element $\lambda$ in the Kashiwara basis has a weight $wt(\lambda)$, with a content $\Bar{c}=(c_0,\dots,c_n)$, such that $\lambda=\Lambda_n-\sum_{i=1}^{n}c_i\alpha_i\in P(\Lambda_n)$. The module $V(\Lambda_n)$ is generated from $\phi$ by the $f_i^{(k)},i=0,\dots,n,k>0$. If we restrict to the action of $f_i^{(k)},i=0,\dots,n-1$, then $V(\Lambda_n)$ is a direct sum of slices with content $c_n=c$. In a slice $c_n=c$, every   Kashiwara basis element $x$ which comes from another  Kashiwara basis element $y$ by $f_i,i=0,\dots,n-1$ will be in the highest weight module generated by $y$, only a Kashiwara basis element $x$ which comes only by $f_n$ from the slice $c_n=c-1$ does not lie in some larger highest weight module within the slice. We claim that these Kashiwara basis element 
generate the slice. 

   Every slice is generated by the  highest weight modules of a basis.  Since each highest weight module is symmetric around the same point, then every element has a symmetric element at the other end of the highest weight module that it generates.
 The level in the case $\lambda=\Lambda_n$ is equal to 2. Let $[\theta_0,\theta_1,\dots, \theta_{n-1},\theta_n]$ be the   hub of the weight of the lowest degree in a slice, where  $\theta_0+2\theta_1+\dots+ 2\theta_{n-1}=2-2\theta_n$. In the high degree end of the slice we have the hub $[-\theta_0,-\theta_1,\dots, -\theta_{n-1},\theta^*_n]$, where $-\theta_0-2\theta_1-\dots-2\theta_{n-1}=2-2\theta^*_n$. Adding these equation give $0=4-2\theta_n-2\theta^*_n$, so $\theta^*_n=2-\theta_n$. According to Lemma 4.1 $V(\Lambda)$ is isomorphic to $V(-\sigma\Lambda)$, so they are symmetric around the average of the previous weights $[0,\dots,0,1]$. 
For $c=1$, we have a single point $[0,\dots,0,1]^2$.
  In our crystal we should get the hub $[0,\dots,0,1]^{2c}$ because $\lambda=\Lambda_n+c\delta \in P(\Lambda)$. The length of $i$-strings for each $i$ around the  hub $[0,\dots,0,1]^{2c}$  in the face is even because of having 0 in each component of the hub, in addition to that  $\tau [0,\dots,0,1]=[0,\dots,0,1]$ and it is the only fixed point.

\end{proof}

	\section{THE CANONICAL BASIS FOR $\Lm=\Lm_n$}

Let $\mathfrak{g}=A^{(2)}_{2n}$ be the affine Lie algebra.
 For an affine Lie algebra given by a symmetrizable Cartan matrix $C=[a_{ij}]$, the quantum enveloping algebra $\mathcal U=  U_q(\mathfrak{g})$ is generated by $e_j$, $f_j$, $K_j$, $K_j^{-1}$, $D$, subject to relations depending on the Cartan matrix [\cite{KMOY},3].

For each residue $i$, we have a power of $q$ depending on the length of the corresponding root in the symmetric power. For our case this gives
\[
q_i=\begin{cases}
	q, &i=0\\q^2, &1\leq i < n \\q^4,&i=n\\
\end{cases}
\]
We define the quantum factorials by 
\[
[k]_i=\frac{q_i^k-q_i^{-k}}{q_i-q_i^{-1}}, [k]_i!=[k]_i\cdot\cdot\cdot[2]_i[1]_i,
\]
and the divided powers, we define by 
 
\[
e_i^{(k)}=\frac{e_i^k}{[k]_i!},f_i^{(k)}=\frac{f_i^k}{[k]_i!}
\]

\subsection{Fock space}
The $q$-Fock space $\mathfrak{F}$ of level 2 is a vector space over $\mathbb{Q}(q)$ with a basis labelled by partitions where we denote by $\lvert \mu \rangle$ the basis element labelled by $\mu$, and the empty multipartition is represented by $\lvert ~\rangle$. This vector space is a module for $\mathcal U$ with combinatorial rules for the divided powers $e_i^{(k)} $ and $f_i^{(k)}$ as follow.

			\begin{itemize}
				\item  For an $i$-addable node $\mathfrak n$, let us n define  $N(\mathfrak n,i)=\#\{$   addable  $i$-nodes below ~$\mathfrak n\}-\# \{$ removable $i$-nodes below $\mathfrak n \}$ and set
				\[
				f_i(\lvert\lambda\rangle)=\sum_{\mathfrak n}q^{N(\mathfrak n,i)}\lvert\lambda^{\mathfrak n}\rangle.
				\]
				
				\item  For an $i$-removable node $\mathfrak m$, let us 
				define $M(\mathfrak m,i)=\#\{$ removable $i$-nodes above~$\mathfrak m\}-  \# \{$  addable  $i$-nodes above $\mathfrak m\}$.
				\[
				e_i(\lvert\mu\rangle)=\sum_{\mathfrak m}q^{M(\mathfrak m,i)}\lvert\mu_{\mathfrak m}\rangle.
				\]
				\item Letting $N(i)=\#\{$   addable  $i$-nodes $\}-\# \{$ removable $i$-nodes $\}$, we let $K_i=q^{h_i}$, act on an element $\lvert\mu\rangle$  of the natural basis by multiplication by $q^{N(i)}$, giving $q^{N(i)}\lvert \mu\rangle$.
				\item Letting $N_0$ be the number of $0$-nodes in $\mu$, $q^d\lvert \mu \rangle=q^{N_0}\lvert \mu \rangle.$
			\end{itemize}

           The submodule of $\mathfrak {F}$ generated by the empty partition $\phi$, under the action of $f_i^{(k)}$ is isomorphic to the irreducible highest weight $\mathcal U$-module $V(\Lambda_n)$, and  contains a canonical basis.
More details can be found in [AS1] and [LT].

			\subsection{Canonical basis elements}
		
		Canonical basis elements lie in Fock spaces, which are modules over the quantum enveloping algebras. For a type A, a good reference for this material which does not require using infinite wedge products is the work of Lascoux, Leclerc and Thibon, [\S 4,\cite{ LLT}]. Working in level $1$, their Fock spaces have a natural basis given by partitions.  This was extended by Ariki and Mathas to higher levels, using multipartitons as labels for the natural basis of the Fock space\cite{AM}.  Our work is combinatorial and does not go into the algebra involved, so our treatment here is quite abbreviated.

		There is an involution of the quantum enveloping algebra called the bar-involution which fixes $e_i$, $f_i$, sends $K_i$ to $K_i^{-1}$, $q^{d}$ to $q^{-d}$, and interchanges $q$ and $q^{-1}$.

        For each h-regular partition $\mu$, there is an unique vector $G(\mu)\in \mathcal F$ that is bar-invariant under the bar involution [and is congruent to the natural basis element of $\lvert \mu \rangle$ modulo $q$]. If when we write $G(\mu)=\sum_{\lambda }d_{\lambda\mu}\lambda$, the coefficient $d_{\mu\mu}$ equals 1 while all the other coefficient $d_{\lambda\mu}$ are polynomials divisible by $q$, the vector $G(\mu)$ is called the  canonical basis vector corresponding to $\mu$ and the set $\{G(\mu)|\mu\in P_h\}$ is the canonical basis. An algorithm for constructing the $G(\mu)$ recursively in the case of partitions of type A was given originally by  Lascoux, Leclerc, and Thibon in \cite{LLT}. The LLT algorithm was extended to higher levels by Jacon \cite{J} using a twisted product of level $2$ Fock spaces, and by Yvonne \cite{Y} to the entire Fock space.  A faster algorithm for an untwisted product of level $2$ Fock spaces is given by Fayers in \cite{F} and implemented by Scrimshaw as the function Fock\_space() in SageMath \cite{SM}.
		The algorithm is recursive in the degree of the multipartition.

\begin{prop}
   For $\Lambda=\Lambda_n, h=2n+1$, if an element $\lvert \lambda \rangle$ appears in a canonical basis element $G(\mu)$, then $\lvert \lambda^t \rangle$ also appears in $G(\mu)$.
\end{prop}

\begin{proof}
In degree $0$, $G(\phi)$ is $\lvert  \rangle$ which is symmetric. Let $G(\mu)=\sum_{\lambda}d_{\lambda\mu}\lambda$ be a canonical basis elements, and let $\supp G(\mu)=\{\lambda|d_{\lambda\mu}\neq0\}$ be the set of partitions in $G(\mu)$ with non-zero cofficient. By the induction 
hypothesis let the partitions $\lvert \lambda \rangle$ and $\lvert \lambda^t \rangle$  appears in $\supp G(\mu)$, such that for $i=1,2\dots,k$,  $\lvert \lambda \rangle \in f_i^{(k)}(G(\mu))$, and we want to prove that $f^{(k+1)}(G(\mu))$ contain $\lvert \lambda^x \rangle$ and $\lvert (\lambda^x)^t \rangle$, where $x$ is the $i$ addable node and $\lvert \lambda^x \rangle\in f^{(k+1)}(G(\mu))$.

Let the node $(s,k)$ be an $i$-addable node of $\lvert \lambda \rangle$, then the node $(k,s)$ is also an $i$-addable node of $\lvert \lambda^t \rangle $, $\forall k,s \in \mathbb{N}$, since they have the same residue according to Lemma 3.1, so in this way we keep $\lvert \lambda^x) \rangle\ $ symmetric to 
$\lvert (\lambda^x)^t \rangle\ $,  after dividing by $[k+1]_{i}$ the coefficient will still non zero.
\end{proof}

\begin{examp}
	Let $h=7$, $\Lm=\Lm_3$.  We calculate the crystal for $c_3 \leq 1$. Each slice with fixed value of $c_3=c$ turns out to be centrally symmetric around the point with hub $[0,0,0,1]$ and defect $2c$, so for $c_3=2$ it is enough to consider the upper half of the slice. The only point in the face with $c_3=0$ is the empty partition because for $c_3=0$, our partitions are all $3$-corner partitions, so we cannot begin to build partitions until $c_3 \geq 1$. Thus we will begin with $c_3=1$. We will list the canonical basis elements by degree and content, giving the hub, the content and the defect, and the canonical basis element. The central symmetry matches a hub $[\theta_0,\dots,\theta_{n-1},\theta_n]$ to $[-\theta_0,-\theta_1,\dots,-\theta_{n-1},2-\theta_n]$  and the defects are preserved.

	\noindent \textbf{Slice $c_3=1$}
\begin{enumerate}
\item 
\begin{itemize}
	\item (0,0,0,1) $[0,0,2,-1]^0$ \young(3)
\end{itemize}	
\item 	
\begin{itemize}
	\item (0,0,1,1)   $[0,1,0,0]^1$ \young(3,2)+$q^2$\young(32)
\end{itemize}	
\item 
\begin{itemize}
	\item (0,1,1,1) $[2,-1,1,0]^1$ \young(3,2,1)	+$q^2$\young(321)
	\item (0,0,2,1) $[0,2,-2,1]^0$ \young(32,2)
\end{itemize}	
\item 
\begin{itemize}
	\item (1,1,1,1) $[0,0,1,0]^{1.5}$ \young(3,2,1,0)+$q^2$\young(3210)
	\item (0,1,2,1) $[2,-1,-1,1]^1$ \young(32,2,1)+$q^2$\young(321,2)
\end{itemize}	
\item 
\begin{itemize}
	\item (2,1,1,1) $[-2,1,1,0]^{1}$ \young(3,2,1,0,0)	+$q^2$\young(32100)
	\item (1,1,2,1) $[0,1,-1,1]^{1.5}$ \young(32,2,1,0)	+$q^2$\young(3210,2)
	\item (0,2,2,1) $[4,-2,0,1]^0$ \young(321,2,1)
\end{itemize}	
\item 
\begin{itemize}
	\item (2,2,1,1) $[0,-1,2,0]^{1}$ \young(3,2,1,0,0,1)	+$q^2$\young(321001)
	\item (2,1,2,1) $[-2,2,-1,1]^{1}$ \young(32,2,1,0,0)	+$q^2$\young(32100,2)
	\item (1,2,2,1) $[2,-1,0,1]^{1.5}$ \young(321,2,1,0)+$q^2$\young(3210,2,1)
\end{itemize}	
\item 
 (2,2,2,1) $[0,0,0,1]^2$
\begin{itemize}
	\item  \young(321,2,1,0,0)+$q$ \young(3210,2,1,0)+$q^4$\young(32100,2,1)
	\item  \young(3,2,1,0,0,1,2)+$q^2$\young(32,2,1,0,0,1)+$q^2$\young(321001,2)	+$q^4$\young(3210012)
	\item \young(32,2,1,0,0,1)+ $q^2$ \young(321,2,1,0,0)	+$q^2$\young(32100,2,1)	+$q^4$\young(321001,2)

\end{itemize}	
\item 
\begin{itemize}
	\item  (2,3,2,1) $[2,-2,1,1]^{1}$: \young(321,2,1,0,0,1)+ 	+$q^2$\young(321001,2,1)
	\item  (2,2,3,1) $[0,1,-2,2]^1:$ \young(32,2,1,0,0,1,2)+$q^2$\young(3210012,2)	
	\item  (3,2,2,1) $[2,-1,0,1]^{1.5}:$ \young(3210,2,1,0,0)+$q^2$ \young(32100,2,1,0)
\end{itemize}	
\end{enumerate}	

\end{examp}

\begin{lem}\label{defect}For a dominant integral weight 
	$\Lambda=a_0\Lambda_0+a_1\Lambda_1+\dots+a_n \Lambda_{n}$, if $\eta=\Lambda-\alpha$ for $\alpha \in Q_+$ has a nonnegative $i$-component $w=\theta_i$ in $\hub(\eta)$, then $\eta, \eta-\alpha_i,\dots,\eta-w\alpha_i$ are vertices in an $i$-string with $s_i(\eta)=\eta-w\alpha_i$. We recall that the symmetrizing matrix $D=\diag(d_0,d_1,\dots,d_n)$ where $d_0=\frac{1}{2}, d_1=1,\dots, d_{n-1}=1,d_n=2$,
	\begin{enumerate}
		\item   The defect of $\lambda=\eta-k\alpha_i$ for $0 \leq k \leq w$ is
		$	\defect(\eta)+d_ik(w-k)$.
		\item  The $i$-components of the hub descend by $2$ as we move down the $i$-string,  so the absolute values of the hub components
		 $\theta_i=\langle h_i, \lambda  \rangle$ decrease as the defect increases and then increase as the defect decreases.
	\end{enumerate}	 
			  In summary,
		\begin{center}
			\begin{tabular}{ | l || c | c | c | c | c | c | }
				\hline
				$\lambda$ &$\eta$ & $\eta-\a_i$&$\dots$ & $\eta-k\a_i$ &$\dots$ & $\eta-w\a_i$ \\ \hline \hline
				$\defect(\lambda)$ &$\defect(\eta)$&$\defect(\eta)+d_i(w-1)$&$\dots$&$\defect(\eta)+d_ik(w-k)$&$\dots$ & $\defect(\eta)$\\ \hline
				$\langle h_i,\lambda \rangle$ &w&w-2&\dots&w-2k&\dots&-w \\ \hline					
				\hline
			\end{tabular}
		\end{center}
\end{lem}

\begin{prop}
For $\Lambda=\Lambda_n, h=2n+1$. Let $\lambda$ be a partition of defect 0 and let $i$ be a residue for which $\lambda$ has $w$ addable $i$-nodes and no removable nodes. Then for any $k$, $0 \leq k \leq w$, the highest power of $q$ in $f_i^{(k)}(\lambda)$ is $2\lfloor \defect(f_i^{(k)}(\lambda))\rfloor$.

\end{prop}

\begin{proof}
Let $\lambda$ be a partition in $P(\Lambda)$ such that, $\hub(\lambda)=[\theta_0,\dots,\theta_i,\dots,\theta_n]$, and $\defect(\lambda)=0$. For a residue $0\leq i \leq n$, let $\theta_i>0$ and suppose $\lambda$ has only $i$ addable $i$-nodes. In this case, $G(\lambda)$ contains one partition multiplied by  $q^0$. According to the previous proposition the length of the $i$-string is 
$w=\theta_i$. since there is one partition in $G(\lambda)$ that is $q^0\lambda$, which means the biggest power of $q$ is 0. That equals  $2\lfloor \defect(\lambda)\rfloor=0$. Let us prove the proposition by induction for $f^{(k)}_i(\lambda)$. For $k=1$ we have $ f_i(\lambda)$, and there is three cases: 
\begin{enumerate}
   
   \item For $ 1\leq i < n$, 
  for $f_i(\lambda)$ we operate by $f_i$ at one partition $q^0\lambda$, according to the definition of $f_i$, in order to get the biggest power of $q$ we should add the $i$-addable node in the top of the partition, where all the other addable nodes are below this node. In this case we have $(w-1)$ addable $i$-nodes, so we multiply  $q^0\lambda$ by $q_i^{w-1}$, where $q_i=q^2$, so we get $q^{2(w-1)}$, this power is equal to $2\lfloor \defect(f_i(\lambda))\rfloor$, where $\lfloor\defect(f_i(\lambda))\rfloor=\defect(f_i(\lambda))=d_i(w-1)=w-1$, according to Lemma 5.1, where $d_i=1$.
  \item For $ i=n$,
  for $f_i(\lambda)$ we operate by $f_i$ at one partition $q^0\lambda$, according to the definition of $f_i$, in order to get the biggest power of $q$ we should add the $i$-addable node in the top of the partition, where all the other addable nodes are below this node. In this case we have $(w-1)$ addable $n$-nodes, so we multiply  $q^0\lambda$ by $q_i^{w-1}$, where $q_i=q^4$, so we get $q^{4(w-1)}$, this power is equal to $2\lfloor \defect(\lambda)\rfloor$, where $\lfloor\defect(\lambda)\rfloor=\defect(\lambda)=d_i(w-1)=2(w-1)$, according to Lemma 5.1, where $d_i=2$.

    \item For $ i=0$,
    for $f_0(\lambda)$ we operate by $f_0$ at one partition $q^0\lambda$, according to the definition of $f_i$, in order to get the biggest power of $q$ we should add the $0$-addable node in the top of the partition, where all the other addable nodes are below this node. In this case we have $(w-2)$ addable $0$-nodes, since according the numbering of nodes in the Young diagram [AS1], there is always a couple of addable 0-nodes next to each other, so we should add the first addable node, in this way the number of addable 0-nodes below is reduced by 1, so we multiply $q^0\lambda$ by $q_i^{w-2}$, where $q_i=q$, so we get $q^{(w-2)}$, this power is equal to $2\lfloor \defect(f_i(\lambda))\rfloor$, where $2\lfloor\defect(f_i(\lambda))\rfloor=2\lfloor d_i(w-1)\rfloor$, according to Lemma 5.1, $d_0=\frac{1}{2}$, so we get $2\lfloor\frac{1}{2}(w-1)\rfloor$, for a hub $[\theta_0,\dots,\theta_n]$, the $\theta_0$ is influenced only by $\alpha_0$ and $\alpha_1$, and we begin the hub  with $\theta_0=0$, so every $\theta_0$ in the hub $(\lambda)\in \hat{P(\Lambda)}$ is even, this gives that $w$ is even, so $2\lfloor\frac{1}{2}(w-1)\rfloor=2(\frac{1}{2}(w-1)-\frac{1}{2})=w-2$.
     \end{enumerate}

    Now we use the induction, as our induction hypothesis, the biggest power of $q_i$ in $f^{(k-1)}_i(\lambda)$ is equal to $2\lfloor \defect(f^{(k-1)}_i\lambda)\rfloor$, where $1< k \leq w$ and suppose we get biggest power of $q$  when we add in the top of the partition.  
   \begin{enumerate}
  \item For $ 1\leq i < n$,
$f^{(k)}_i(\lambda)=f_i(f^{(k-1)}_i(\lambda)).[k]^{-1}_{q_i}$, so we want to operate by $f_i$ on $f^{(k-1)}_i(\lambda)$, as we suppose, we get the biggest power of $q$ when we add in the top of the partition, where the addable nodes are below the addable node, so we should multiply the biggest power of $f^{(k-1)}_i(\lambda)$ by $q_i^{w-k}$, according to the 
induction hypothesis the biggest power on $f^{(k-1)}_i(\lambda)$ is $2\lfloor \defect(f^{(k-1)}_i(\lambda)\rfloor$ that in our case equal to $2\lfloor d_i(k-1)(w-(k-1))\rfloor$ according to Lemma 5.1. Since $d_i=1$ we can write the $2\lfloor \defect(f^{(k-1)}_i\lambda)\rfloor$ as  $2[(k-1)(w-(k-1))]$, $q_i=q^2$, so all told $q^{2[(w-k)+(k-1)(w-(k-1))]}=q^{2k(w-k)+2(k-1)}$. Now we should divide by $[k]_{q_i}$ $[F]$, so we get $q^{2k(w-k)}$, this power is equal to $2\lfloor \defect(f^{(k)}_i(\lambda)\rfloor=2\lfloor d_ik(w-k) \rfloor=2k(w-k)$,  where $d_i=1$ according to Lemma 5.1.

\item For $ i=n$, 
$f^{(k)}_i(\lambda)=f_i(f^{(k-1)}_i(\lambda)).[k]^{-1}_{q_i}$, so we want to operate by $f_i$ on $f^{k-1}_i(\lambda)$, as we suppose, we get the biggest power of $q$ when we add in the top of the partition, where the addable nodes are below the addable node, so we should multiply the biggest power of $f^{(k-1)}_i(\lambda)$ by $q_i^{w-k}$, according to the 
induction hypothesis the biggest power on $f^{k-1}_i(\lambda)$ is $2\lfloor \defect(f^{k-1}_i(\lambda)\rfloor$ that in our case equal to $2\lfloor d_i(k-1)(w-(k-1))\rfloor$ according to Lemma 5.1. Since $d_i=2$ we can write $2\lfloor \defect(f^{k-1}_i(\lambda)\rfloor$ as $4(k-1)(w-(k-1))$, $q_i=q^4$, so all told $q^{4[(w-k)+(k-1)(w-(k-1))]}=q^{4k(w-k)+4(k-1)}$. Now we should divide by $[k]_{q_i}$ $[F]$, so we get $q^{4k(w-k)}$, this power is equal to $2\lfloor \defect(f^{(k)}_i(\lambda)\rfloor=2\lfloor d_ik(w-k) \rfloor=4k(w-k)$,  where $d_i=2$ according to Lemma 5.1.

   \item For $ i=0$,
   $f^{(k)}_i(\lambda)=f_i(f^{(k-1)}_i(\lambda))$, so we want to operate by $f_i$ on $f^{(k-1)}_i(\lambda)$, as we suppose, we get the biggest power of $q$ when we add in the top of the partition, where the other addable nodes are below the node we add. For $i=0$ we need to deal with two cases:
   \begin{enumerate}
   
   \item For odd $k$, 
since according the numbering of nodes in the Young diagram [AS1], there is always a couple of addable 0-nodes next to each other, so if we operate by $f_0^{k-1}$, that says we operate by $f_0$ an even number of times, and the 0-addable node we should add is the first 0-node, in this state the number of the addable 0-nodes below this node is $w-k-1$, which says we multiply the biggest power of $f^{(k-1)}_i(\lambda)$ by $q_i^{w-k-1}$. According to the induction hypothesis, the biggest power on $f^{(k-1)}_i(\lambda)$ is $2\lfloor \defect(f^{(k-1)}_i(\lambda)\rfloor$ which according to Lemma 5.1 is equal to $2\lfloor d_i(k-1)(w-(k-1))\rfloor$, where $d_i=\frac{1}{2}$ .

We can write $\lfloor d_i(k-1)(w-(k-1))\rfloor=\frac{(k-1)(w-(k-1))}{2}$ since we showed that $w$ is even and also $k-1$ is even, so all told $q_i^{w-k-1+2(\frac{(k-1)(w-(k-1))}{2}}=q_i^{k(w-k)+(k-2)}$, $q_i=q$ and we should divide by $[k]_{q_i}$, so we get $q_i^{k(w-k)-1}$, this power is equal to $2\lfloor \defect(f_i^{(k)}(\lambda))\rfloor$, since,

\ $2\lfloor \defect(f^{(k)}_i(\lambda))\rfloor$$ =2\lfloor d_ik(w-k) \rfloor=2\lfloor \frac{k(w-k)}{2} \rfloor=2(\frac{k(w-k)}{2}-\frac{1}{2})=k(w-k)-1$.

   \item For even $k$,
   since according the numbering of nodes in the Young diagram [AS1], there is always a couple of addable 0-nodes next to each other, so if we operate by $f_0^{(k-1)}$, that says we operate by $f_0$ an odd number times, and the 0-addable node we should add is the second 0-node, in this state the number of the addale 0-nodes below this node is $w-k$ and there is 1 removable 0-node next to this 0-addable node,  which says we multiply the biggest power of $f^{(k-1)}_i(\lambda)$ by $q_i^{w-k-1}$. 
   
   According to the induction hypothesis, the biggest power on $f^{(k-1)}_i(\lambda)$ is $2\lfloor \defect(f^{(k-1)}_i(\lambda)\rfloor$ which according to Lemma 5.1 is equal to $2\lfloor d_i(k-1)(w-(k-1))\rfloor$, where $d_i=\frac{1}{2}$ . We can write $\lfloor d_i(k-1)(w-(k-1))\rfloor=\frac{(k-1)(w-(k-1))}{2}-\frac{1}{2}$ since we showed that $w$ is even and  $k-1$ is odd, so all told $q_i^{w-k-1+2(\frac{(k-1)(w-(k-1))}{2}-\frac{1}{2})}=q_i^{k(w-k)+(k-3)}$, $q_i=q$, and according to [AS1] since we have a couple of 0-addable node we should multiply by $(1+q^2)$ so the highest power of $q_i$ rise by 2, and we get $q_i^{k(w-k)+(k-1)}$  and we should divide by $[k]_{q_i}$, so we get $q^{k(w-k)}$, this power is equal to $2\lfloor \defect(f^{(k)}_i(\lambda))\rfloor$, since,
   $2\lfloor \defect(f^{(k)}_i(\lambda))\rfloor=2\lfloor d_ik(w-k) \rfloor=2\lfloor \frac{k(w-k)}{2} \rfloor=2(\frac{k(w-k)}{2})=k(w-k)$.
   \end{enumerate}

\end{enumerate}

\end{proof}

\noindent Keywords:  affine Lie algebra, highest weight representation, spin blocks, cyclotomic Hecke algebras

\noindent Conflict of Interest: None\\
\noindent Data availability: No data\\
\noindent Ola Amara-Omari
ORCID: 0000-0002-0651-5123

\end{document}